\documentclass[10pt,a4paper]{amsart}

\usepackage[margin=22mm]{geometry}
\usepackage[T1]{fontenc}
\usepackage{lmodern}
\usepackage{amsmath,amssymb,mathtools}
\usepackage{microtype}
\usepackage[hidelinks]{hyperref}
\hypersetup{
  pdftitle={A negative answer to the Erdos-Sarkozy question},
  pdfauthor={Simone Costa},
  pdfsubject={Subset sums without three-term arithmetic progressions},
  pdfkeywords={subset sums, arithmetic progressions, Erdos-Sarkozy question, subset-sum-distinct sets}
}

\newtheorem{theorem}{Theorem}[section]
\newtheorem{proposition}[theorem]{Proposition}
\newtheorem{corollary}[theorem]{Corollary}
\theoremstyle{remark}

\newcommand{\Z}{\mathbb Z}
\newcommand{\ee}{\varepsilon}

\title{A negative answer to the Erd\H{o}s--S\'ark\H{o}zy question}
\author{Simone Costa}
\address{DICATAM, Universit\`a degli Studi di Brescia, Via Branze 43, 25123 Brescia, Italy}
\email{simone.costa@unibs.it}
\date{}

\begin{document}

\begin{abstract}
For a finite set $A$ of positive integers, let $H(A)$ be its set of subset sums, and let $g_3(n)$ be the least $N$ for which some $n$-element set $A\subseteq\{1,\ldots,N\}$ has $H(A)$ free of nonconstant three-term arithmetic progressions. Erd\H{o}s and S\'ark\H{o}zy asked whether $g_3(n)\gg 3^n$. We prove
\[
  \liminf_{n\to\infty}\frac{g_3(n)}{3^n}=0.
\]
More precisely, for every $\ee>0$ there is an integer $d\ge2$ such that
$g_3(d\ell)\le \ee 3^{d\ell}$ for all sufficiently large $\ell$.
The proof uses Korsky's characterization of the problem in terms of ternary coefficient sums and a consequence of an OpenAI construction that provides positive integer coefficients whose linear form is injective on large integer boxes. A base-three expansion then gives the result.
\end{abstract}

\maketitle

\section{Introduction}

For a finite set
\[
 A=\{a_1,\ldots,a_n\}\subset\Z_{>0},
\]
write
\[
 H(A)=\left\{\sum_{i=1}^n \delta_i a_i:\delta_i\in\{0,1\}\right\}.
\]
Here and below, a nonconstant three-term arithmetic progression is a triple $x,y,z$ satisfying $x+z=2y$ with $x,y,z$ not all equal. Following Erd\H{o}s and S\'ark\H{o}zy~\cite{ErdosSarkozy1992}, let $g_3(n)$ be the least positive integer $N$ for which there is an $n$-element set $A\subseteq\{1,\ldots,N\}$ such that $H(A)$ contains no nonconstant three-term arithmetic progression.

Erd\H{o}s and S\'ark\H{o}zy proved
\[
 g_3(n)\gg \frac{3^n}{n^C}
\]
for some absolute constant $C>0$, and asked whether
\[
 g_3(n)\gg 3^n.
\]
The powers-of-three construction $A=\{1,3,\ldots,3^{n-1}\}$ gives the elementary upper bound $g_3(n)\le 3^{n-1}$. Indeed, write the base-three digits of three subset sums $x,y,z$ as $u_i,v_i,w_i\in\{0,1\}$. If $x+z=2y$, then there are no carries, so $u_i+w_i=2v_i$ for every $i$. Hence $u_i=v_i=w_i$ for every $i$, and therefore $x=y=z$.

In the terminology introduced by Bae~\cite{Bae2002}, a set $A$ is \emph{$2$-fold subset-sum-distinct} if all sums
\[
 \sum_{a\in A}\varepsilon_a a,\qquad \varepsilon_a\in\{0,1,2\},
\]
are distinct. Korsky~\cite[Proposition~4.1 and Corollary~4.2]{Korsky2026} observed that, for sets of positive integers, this condition is exactly equivalent to $H(A)$ being free of nonconstant three-term arithmetic progressions. Dutta's general estimate for bounded-coefficient dissociated sets~\cite{Dutta2026}, specialized to this case, gives
\[
 g_3(n)\ge
 \left(\frac{\sqrt3}{4\sqrt\pi}+o(1)\right)\frac{3^n}{\sqrt n}.
\]
Korsky sharpened this to the exact finite lower bound
\[
 g_3(n)\ge \frac{T_n-1}{2}+\sum_{j=0}^{n-1}T_j,
 \qquad T_j=[x^j](1+x+x^2)^j,
\]
and hence
\[
 g_3(n)\ge
 \left(\frac{\sqrt3}{2\sqrt\pi}+o(1)\right)\frac{3^n}{\sqrt n}.
\]
The question whether a positive constant multiple of $3^n$ is always necessary remains open in that work and is also recorded as Erd\H{o}s Problem~\#817~\cite{BloomErdos817}.

A separate recent development is an AI-generated proof of Erd\H{o}s Problem~\#1 produced by an OpenAI research system; an accompanying mathematical exposition is available from the Erd\H{o}s Problem~\#1 entry~\cite{BloomErdos1,OpenAI2026}. Korsky's characterization, the OpenAI construction just mentioned, and a base-three expansion give a negative answer to the Erd\H{o}s--S\'ark\H{o}zy question.

\begin{theorem}\label{thm:main}
For every $\ee>0$ there is an integer $d\ge2$ such that
\[
 g_3(d\ell)\le \ee\,3^{d\ell}
\]
for all sufficiently large integers $\ell$.
\end{theorem}

\begin{corollary}\label{cor:liminf}
One has
\[
 \liminf_{n\to\infty}\frac{g_3(n)}{3^n}=0.
\]
In particular, there do not exist constants $c>0$ and $n_0$ such that $g_3(n)\ge c3^n$ for every $n\ge n_0$.
\end{corollary}

\section{Previous results}

We use two results from earlier work. The first is Korsky's characterization of the three-term problem in terms of ternary coefficient sums~\cite[Corollary~4.2]{Korsky2026}.

\begin{proposition}[Korsky]\label{prop:korsky}
For every $m\ge1$,
\[
 g_3(m)=\min\left\{\max_{1\le r\le m}b_r:
 (b_1,\ldots,b_m)\in\Z_{>0}^m,\ 
 \varepsilon\longmapsto\sum_{r=1}^m\varepsilon_r b_r
 \text{ is injective on }\{0,1,2\}^m\right\}.
\]
\end{proposition}

The second result is a consequence of the construction in the OpenAI exposition~\cite[Sections~2--6]{OpenAI2026}. We state only the consequence needed here.

\begin{proposition}[Consequence of the OpenAI construction]\label{prop:openai}
For every $K>0$ there exist integers $n\ge1$, $R\ge1$, $D\ge1$, and $E\ge0$, together with integer coefficients
\[
 a_0(t),\ldots,a_n(t)\qquad(t\in\Z),
\]
such that:
\begin{enumerate}
\item $KD<R^n$;
\item $a_i(t)/t^n\to D$ as $t\to\infty$ for every $0\le i\le n$;
\item for all sufficiently large positive integers $t$, all $a_i(t)$ are positive;
\item if $t,Q$ are positive integers satisfying
\[
 Q\le (t-E)R,
\]
then the map
\[
 \{0,\ldots,Q-1\}^{n+1}\longrightarrow\Z,
 \qquad
 (x_0,\ldots,x_n)\longmapsto\sum_{i=0}^n a_i(t)x_i
\]
is injective.
\end{enumerate}
\end{proposition}

\begin{proof}
Given $K>0$, choose an integer $k>K$. In~\cite[equation~(16)]{OpenAI2026}, the construction gives integers $n,R,D$ with
\[
 kD<R^n,
\]
which implies (1). Equation~(22) in the same source gives (2), and since $D>0$ it also gives (3) for all sufficiently large positive $t$. Finally, Proposition~5.2 there allows any positive integer $q_0$ satisfying
\[
 q_0\le (t-E_B)R.
\]
Since $q_0>0$ and $R>0$, this inequality also implies $t>E_B$, as required in that proposition. Equation~(21) describes exactly which integer vectors make the linear form vanish. The argument following equation~(24) then shows that two vectors in $\{0,\ldots,q_0-1\}^{n+1}$ with the same value of the linear form must be equal; this is equation~(25). Taking $E=E_B$ and $Q=q_0$ gives (4).
\end{proof}

\section{Base-three expansion}

We now apply Propositions~\ref{prop:korsky} and~\ref{prop:openai}. We take $Q=3^\ell$, so each coordinate ranges from $0$ to $3^\ell-1$, and use the usual base-three representation of these integers.

\begin{proof}[Proof of Theorem~\ref{thm:main}]
Fix $\ee>0$. Choose $K>1/(3\ee)$ and apply Proposition~\ref{prop:openai}. Keep the resulting integers $n,R,D,E$ fixed, and put
\[
 d=n+1.
\]
For a positive integer $\ell$, set
\[
 Q_\ell=3^\ell,
 \qquad
 t_\ell=E+\left\lceil\frac{3^\ell}{R}\right\rceil.
\]
Then
\[
 Q_\ell=3^\ell\le(t_\ell-E)R.
\]
Since $t_\ell\to\infty$, for all sufficiently large $\ell$ the integers
$a_0(t_\ell),\ldots,a_n(t_\ell)$ are positive, and Proposition~\ref{prop:openai}(4) gives that
\begin{equation}\label{eq:box3}
 (x_0,\ldots,x_n)\longmapsto
 \sum_{i=0}^n a_i(t_\ell)x_i
\end{equation}
is injective on $\{0,\ldots,3^\ell-1\}^{n+1}$.

Define
\[
 A_\ell=
 \left\{3^j a_i(t_\ell):0\le i\le n,\ 0\le j<\ell\right\}.
\]
We claim that the map
\[
 \{0,1,2\}^{(n+1)\ell}\longrightarrow\Z,
 \qquad
 (\delta_{i,j})\longmapsto
 \sum_{i=0}^n\sum_{j=0}^{\ell-1}\delta_{i,j}3^j a_i(t_\ell)
\]
is injective. Suppose that two choices of coefficients $\delta_{i,j},\delta'_{i,j}\in\{0,1,2\}$ give the same sum, and set
\[
 x_i=\sum_{j=0}^{\ell-1}\delta_{i,j}3^j,
 \qquad
 x'_i=\sum_{j=0}^{\ell-1}\delta'_{i,j}3^j.
\]
Then $0\le x_i,x'_i\le3^\ell-1$ and
\[
 \sum_{i=0}^n a_i(t_\ell)x_i
 =\sum_{i=0}^n a_i(t_\ell)x'_i.
\]
Injectivity in~\eqref{eq:box3} gives $x_i=x'_i$ for every $i$. Uniqueness of base-three expansion then gives $\delta_{i,j}=\delta'_{i,j}$ for every $i,j$. Thus the map is injective. Taking coefficient arrays with one entry equal to $1$ and all the others equal to $0$ shows in particular that the $(n+1)\ell$ integers $3^j a_i(t_\ell)$ are pairwise distinct, so
\[
 |A_\ell|=(n+1)\ell=d\ell.
\]
Proposition~\ref{prop:korsky} therefore gives
\begin{equation}\label{eq:g3bound}
 g_3(d\ell)\le \max A_\ell.
\end{equation}

It remains to estimate the largest element. Since the number of coefficients is fixed and
$a_i(t)/t^n\to D$ for every $i$,
\[
 \max_{0\le i\le n}a_i(t_\ell)=(D+o(1))t_\ell^n
 \qquad(\ell\to\infty).
\]
Moreover
\[
 t_\ell=\frac{3^\ell}{R}+O(1),
 \qquad
 t_\ell^n=\frac{3^{n\ell}}{R^n}(1+o(1)).
\]
Hence
\begin{align*}
 \max A_\ell
 &\le 3^{\ell-1}\max_i a_i(t_\ell)\\
 &=\left(\frac{D}{3R^n}+o(1)\right)3^{(n+1)\ell}.
\end{align*}
By Proposition~\ref{prop:openai}(1),
\[
 \frac{D}{3R^n}<\frac1{3K}<\ee.
\]
Combining this with~\eqref{eq:g3bound} yields, for all sufficiently large $\ell$,
\[
 g_3(d\ell)\le\ee\,3^{d\ell}.
\]
\end{proof}

\begin{proof}[Proof of Corollary~\ref{cor:liminf}]
Given $\ee>0$, Theorem~\ref{thm:main} provides infinitely many integers $m=d\ell$ for which
$g_3(m)/3^m\le\ee$. Since $\ee$ is arbitrary and the ratio is nonnegative, its lower limit is zero.
\end{proof}

\section*{Acknowledgements}
The author used ChatGPT (OpenAI, GPT-5.6 Sol, September 2026) to identify and check the base-three application of Proposition~\ref{prop:openai} to the Erd\H{o}s--S\'ark\H{o}zy question, to check calculations and references, and to assist in revising the manuscript. This use is separate from the OpenAI-generated construction cited in~\cite{OpenAI2026}. The author checked the mathematical statements and references and takes full responsibility for the paper.

\end{document}